\documentclass[11pt]{article}

\usepackage{url}
\usepackage{epsfig}
\usepackage{graphicx}
\usepackage{amsmath,amsthm,paralist}
\usepackage{amssymb}
\usepackage{epstopdf}
\usepackage{cite}
\usepackage{citesort}
\usepackage{graphicx,amsmath,amsthm,amsbsy,amssymb,citesort,epsfig}

\def\real    { \mathbb{R} }
\def\reals   { \mathbb{R} }

\newtheorem{lemma}{Lemma}%[section] %%    with section number.
\newtheorem{definition}{Definition}%[section]

\newtheorem{theorem}{Theorem}

\newcommand{\bitem}{\begin{itemize}}
\newcommand{\eitem}{\end{itemize}}

\newcommand{\supp}{\mathrm{supp}}
\newcommand{\beqn}{\begin{equation}}
\newcommand{\eeqn}{\end{equation}}
\newcommand{\balign}{\begin{align}}
\newcommand{\ealign}{\end{align}}

\def\x{{\mathbf x}}
\def\y{{\mathbf y}}

\def\u{{\mathbf u}}
\def\v{{\mathbf v}}
\def\z{{\mathbf z}}

\def\w{{\mathbf w}}
\def\z{{\mathbf z}}

\def\A{{\mathbf A}}
\def\bI{{\mathbf I}}
\def\I{{\mathbf I}}

\def\bPhi{\boldsymbol{\Phi}}
\def\Proj{{\mathbf P}}

\newcommand{\vs}{\ensuremath{\mathbf s}}

\newcommand{\vx}{\ensuremath{\mathbf x}}

\def \Ran {{\mathcal{R}}}

\title{\textbf{A Simple Proof that Random Matrices are Democratic}}

\author{\em Mark A.\ Davenport, Jason N.\ Laska, Petros T.\ Boufounos,\\
\em and Richard G.\ Baraniuk%
\thanks{MAD, JNL, and RGB are with the Department of Electrical and
Computer Engineering, Rice University, Houston, TX, USA. Email: \{md,
laska, richb\}@rice.edu. \protect\\
PTB is with Mitsubishi Electric Research Laboratories (MERL), Boston, MA, USA. Email: petrosb@merl.com. \protect\\
\indent This work was supported by the grants NSF CCF-0431150, CCF-0728867,
CNS-0435425, and CNS-0520280, DARPA/ONR N66001-08-1-2065, ONR
N00014-07-1-0936, N00014-08-1-1067, N00014-08-1-1112, and
N00014-08-1-1066, AFOSR FA9550-07-1-0301, ARO MURI W311NF-07-1-0185,
and the Texas Instruments Leadership University Program.
    \protect\\ \indent Web: dsp.rice.edu/cs} \protect\\\protect\\
    Rice University \protect\\ Department of Electrical and Computer Engineering\protect\\
    Technical Report TREE0906}

\date{November 4, 2009}

\begin{document}

\maketitle

\vspace{-0.3in}

\begin{abstract}
The recently introduced theory of \emph{compressive sensing} (CS) enables the reconstruction of sparse or compressible signals from a small set of nonadaptive, linear measurements.   If properly chosen, the number of measurements can be significantly smaller than the ambient dimension of the signal and yet preserve the significant signal information.  Interestingly, it can be shown that random measurement schemes provide a near-optimal encoding in terms of the required number of measurements.  In this report, we explore another relatively unexplored, though often alluded to, advantage of using random matrices to acquire CS measurements.  Specifically, we show that random matrices are {\em democractic}, meaning that each measurement carries roughly the same amount of signal information.  We demonstrate that by slightly increasing the number of measurements, the system is robust to the loss of a small number of arbitrary measurements. In addition, we draw connections to oversampling and demonstrate stability from the loss of significantly more measurements.
\end{abstract}

\section{Introduction}
The recently developed {\em compressive sensing} (CS) framework allows us to acquire a signal $\x \in \mathbb{R}^N$ from a small set of $M$ non-adaptive, linear measurements~\cite{Can::2006::Compressive-sampling,Don::2006::Compressed-sensing}.  This process can be
represented as
\begin{equation} \label{eq:yphix}
\y=\bPhi \x
\end{equation}
where $\bPhi$ is an $M \times N$ matrix that models the measurement system. The hope is that we can design $\bPhi$ so that $\x$ can be accurately recovered even when $M \ll N$.   In general this is not possible, but if $\x$ is $K$-sparse, meaning that it has only $K$ nonzero entries then it is possible to design $\bPhi$ that preserve the information about $\x$ using only $M = O (K \log (N/K))$ measurements.  The most commonly studied $\bPhi$ that satisfy this bound on $M$ are \emph{random}, i.e., each entry of $\bPhi$ is drawn independently from some suitable distribution~\cite{BarDavDeV::2008::A-Simple-Proof}.  We will focus our attention on such $\bPhi$.

Among the advantages of random measurements is a property commonly referred to as \emph{democracy}.  While it is not usually rigorously defined in the literature, democracy is usually taken to mean that each measurement contributes a similar amount of information about the signal $\x$ to the compressed representation $\y$~\cite{CalDau::2002::The-pros-and-cons,Gun::2000::Harmonic-analysis,Can::2005::Integration-of-sensing}.\footnote{The
original introduction of this term was with respect to quantization~\cite{CalDau::2002::The-pros-and-cons,Gun::2000::Harmonic-analysis},
i.e., a democratic quantizer would ensure that each bit is given
``equal weight.''  As the CS framework developed, it became
empirically clear that CS systems exhibited this property with
respect to compression~\cite{Can::2005::Integration-of-sensing}.}  Others have described democracy to mean that each measurement is equally important (or unimportant)~\cite{RomWak::2007::Compressed-sensing:}.   Despite the fact that democracy is so frequently touted as an advantage of random measurements, it has received little analytical attention in the CS context.  Perhaps more surprisingly, the property has not been explicitly exploited in applications until recently~\cite{LasBouDav::2009::Demcracy-in-action:}.

The fact that random measurements are democratic seems intuitive; when using random measurements, each measurement is a randomly weighted sum of a large fraction (or all) of the entries of $\x$, and since the weights are chosen independently at random, no preference is given to any particular entries.  More concretely, suppose that the measurements $y_1, y_2, \ldots, y_M$ are independent and identically distributed (i.i.d.) according to some distribution $f_Y$, as is the case for the $\bPhi$ considered in this report.  Now suppose that we select $\widetilde{M} < M$ of the $y_i$ at random (or according to some procedure that is {\em independent} of $\y$).  Then clearly, we are left with a length-$\widetilde{M}$ measurement vector $\widetilde{\y}$ such that each $\widetilde{y}_i \sim f_Y$.  Stated another way, if we set $D = M - \widetilde{M}$, then there is no difference between collecting $\widetilde{M}$ measurements and collecting $M$ measurements and deleting $D$ of them, provided that this deletion is done independently of the actual values of $\y$.

However, following this line of reasoning will ultimately lead to a rather weak definition of democracy.  To see this, consider the case where the measurements are deleted by an adversary. By adaptively deleting the entries of $\y$ one can change the distribution of $\widetilde{\y}$.  For example, the adversary can delete the $D$ largest elements of $\y$, thereby skewing the distribution of $\widetilde{\y}$.  In many cases, especially if the same matrix $\bPhi$ will be used repeatedly with different measurements being deleted each time, it would be far better to know that {\em any} $\widetilde{M}$ measurements will be sufficient to reconstruct the signal.  This is a significantly stronger requirement.

In order to formally define this stronger notion of democracy, we must first describe the properties that a matrix must satisfy to ensure stable reconstruction.  Towards that end, we recall the definition of the {\em restricted isometry property} (RIP) for the matrix $\bPhi$~\cite{CandesDLP}.

\begin{definition}
A matrix $\bPhi$ satisfies the RIP of order $K$ with constant $\delta \in (0,1)$ if
\begin{equation} \label{eq:RIP}
(1-\delta)\|\x\|_2^2 \le \|\bPhi \x \|^2_2 \le (1+\delta)\|\x\|_2^2
\end{equation}
holds for all $\x$ such that $\| \x\|_0 \le K$.
\end{definition}
Much is known about matrices that satisfy the RIP, but for our purposes it suffices to note that if we draw a random $M \times N$ matrix $\bPhi$ whose entries $\phi_{ij}$ are i.i.d.\ sub-Gaussian random variables, then provided that
\begin{equation}
M = O\left(K\log (N/K) \right), \label{eq:ripm}
\end{equation}
we have that with high probability $\bPhi$ will satisfy the RIP of order $K$ with constant $\delta$~~\cite{BarDavDeV::2008::A-Simple-Proof,DeVoreL1IO}.

When it is satisfied, the RIP for a matrix $\bPhi$ provides a sufficient condition to guarantee successful sparse recovery using a wide variety of algorithms~\cite{CandesDLP,CandesRIP,MIP_cai,DWOMP,romp,romp2,cosamp,sp,Thresh,BluDav::2008::Iterative-hard,chartrand2008restricted}. As an example, the RIP of order $2K$ (with isometry constant $\delta < \sqrt{2}-1$) is a sufficient condition to permit $\ell_1$-minimization (the canonical convex optimization problem for sparse approximation) to exactly recover any $K$-sparse signal and to approximately recover those that are nearly sparse~\cite{CandesRIP}. The same assumption is also a sufficient condition for robust recovery in noise using a modified $\ell_1$-minimization~\cite{CandesRIP}.

The RIP also provides us with a way to quantify our notion of democracy.  To do so, we first establish some notation that will prove useful throughout this report.  Let $\Gamma \subset \{1,2,,\ldots, M\}$.  By $\bPhi^{\Gamma}$ we mean the $|\Gamma| \times M$ matrix obtained by selecting the rows of $\bPhi$ indexed by $\Gamma$.  Alternatively, if $\Lambda \subset \{1,2,\dots,N\}$, then we use $\bPhi_\Lambda$ to indicate the $M \times |\Lambda|$ matrix obtained by selecting the columns of $\bPhi$ indexed by $\Lambda$. Following~\cite{LasBouDav::2009::Demcracy-in-action:}, we now formally define democracy as follows.
\begin{definition}
Let $\bPhi$ be and $M \times N$ matrix, and let $\widetilde{M} \le M$ be given.  We say that $\bPhi$ is $(\widetilde{M},K,\delta)$-democratic if for all $\Gamma$ such that $|\Gamma|\ge \widetilde{M}$ the matrix $\bPhi^{\Gamma}$ satisfies the RIP of order $K$ with constant $\delta$.
\end{definition}

In Section~\ref{sec:simpleproof} below we present a simple proof that Gaussian matrices are democratic and demonstrate how the proof can be extended to sub-Gaussian matrices.  The core of this proof can be found in \cite{LasBouDav::2009::Demcracy-in-action:}, but is included in full in this report.  In Section~\ref{sec:disc} we discuss the implications of the result and alternative interpretations.  Section~\ref{sec:thms} contains the additional theorems required by the proof.

\section{Random matrices are democratic}
\label{sec:simpleproof}

We now demonstrate that certain randomly generated matrices are democratic.  While the theorem actually holds (with different constants) for the more general class of {\em sub-Gaussian} matrices, for simplicity we restrict our attention to Gaussian matrices.  We provide discussion of the sub-Gaussian case in Section~\ref{sec:thms}.

\begin{theorem} \label{thm:democracy}
Let $\bPhi$ by an $M\times N$ matrix with elements $\phi_{ij}$ drawn according to $\mathcal{N}(0, 1/M)$ and let $\widetilde{M} \le M$, $K < \widetilde{M}$, and $\delta \in (0,1)$ be given.  Define $D = M - \widetilde{M}$. If
\begin{align}
\label{eq:Mdem}
M = C_{1} (K + D)\log\left(\frac{N+M}{K + D}\right),
\end{align}
then with probability exceeding $1 - 3 e^{-C_2 M}$ we have that $\bPhi$ is $(\widetilde{M},K,\delta/(1-\delta))$-democratic, where $C_1$ is arbitrary and $C_2  = (\delta/8)^2 - \log(42 e/ \delta)/C_1.$
\end{theorem}
\begin{proof}
Our proof consists of two main steps.  We begin by defining the $M \times (N+M)$ matrix $\A =[\bI~\bPhi]$ formed by appending $\bPhi$ to the $M\times M$ identity matrix.  Theorem~\ref{thm:rip}, also found in \cite{LasDavBar::2009::Exact-signal}, demonstrates that under the assumptions in the theorem statement, with probability exceeding $1 - 3 e^{-C_2 M}$ we have that $\A$ satisfies the RIP of order $K+D$ with constant $\delta$. The second step is to use this fact to show that all possible $\widetilde{M} \times N$ submatrices of $\bPhi$ satisfy the RIP of order $K$ with constant $\delta/(1-\delta)$.

Towards this end, we let $\Gamma \subset \{1, 2, \ldots, M\}$ be an arbitrary subset of rows such that $| \Gamma | \ge \widetilde{M}$.  Define $\Lambda = \{1, 2, \ldots, M \} \setminus \Gamma$ and note that $|\Lambda| = D$.  Additionally, let
\begin{equation} \label{eq:Pdef}
\Proj_{\Lambda} \triangleq \A_{\Lambda} \A_{\Lambda}^{\dagger},
\end{equation}
be the orthogonal projector onto $\mathcal{R}(\A_{\Lambda})$, i.e., the range, or column space, of $\A_{\Lambda}$.\footnote{$\A_\Lambda^\dag = (\A_{\Lambda}^T \A_{\Lambda})^{-1} \A_{\Lambda}^T$ denotes the Moore-Penrose pseudoinverse of $\A_\Lambda$.}  Furthermore, we define
\begin{equation} \label{eq:Pperpdef}
\Proj_{\Lambda}^{\perp} \triangleq \I - \Proj_{\Lambda},
\end{equation}
as the orthogonal projector onto the orthogonal complement of $\mathcal{R}(\A_{\Lambda})$.  In words, this projector nulls the columns of $\A$ corresponding to the index set $\Lambda$.   Now, note that $\Lambda \subset \{1, 2, \ldots, M \}$, so $\A_\Lambda = \I_\Lambda$.  Thus,
$$
\Proj_\Lambda  = \I_\Lambda \I_\Lambda^\dag = \I_\Lambda (\I_\Lambda^T \I_\Lambda)^{-1} \I_\Lambda^T = \I_\Lambda \I_\Lambda^T = \I(\Lambda),
$$
where we use $\I(\Lambda)$ to denote the $M \times M$ matrix with all zeros except for ones on the diagonal entries corresponding to the columns indexed by $\Lambda$.  (We distinguish the $M \times M$ matrix $\I(\Lambda)$ from the $M \times D$ matrix $\I_\Lambda$ --- in the former case we replace columns not indexed by $\Lambda$ with zero columns, while in the latter we remove these columns to form a smaller matrix.)  Similarly, we have
$$
\Proj_\Lambda^\perp = \I - \Proj_\Lambda = \I(\Gamma).
$$
Thus, we observe that the matrix $\Proj_\Lambda^\perp \A = \I(\Gamma) \A$ is simply the matrix $\A$ with zeros replacing all entries on any row $i$ such that $i \notin \Gamma$, i.e., $( \Proj_\Lambda^\perp \A )^\Gamma = \A^\Gamma$ and $( \Proj_\Lambda^\perp \A )^\Lambda = \boldsymbol{0}$.  Furthermore, Theorem~\ref{lem:PRIP}, also found in \cite{SPARS}, states that for $\A$ satisfying the RIP of order  $K+D$ with constant $\delta$, we have that
\begin{align}
\label{eq:projrip}
\left(1 - \frac{\delta}{1-\delta} \right) \|\u\|_2^2 \le \| \Proj_\Lambda^\perp \A \u\|_2^2 \le (1+\delta) \|\u\|_2^2,
\end{align}
holds for all $\u \in \mathbb{R}^{N+M}$ such that $\|\u\|_{0} = K + D - |\Lambda| = K$ and $\supp(\u)\cap\Lambda = \emptyset$.  Equivalently, letting $\Lambda^c = \{1,2, \ldots, N+M\} \setminus \Lambda$, this result states that $(\I(\Gamma) \A)_{\Lambda^c}$ satisfies the RIP of order $K$ with constant $\delta/(1-\delta)$.   To complete the proof, we note that if $(\I(\Gamma) \A)_{\Lambda^c}$ satisfies the RIP of order $K$ with constant $\delta/(1-\delta)$, then we trivially have that $\I(\Gamma) \bPhi$ also has the RIP of order at least $K$ with constant $\delta/(1-\delta)$, since $\I(\Gamma) \bPhi$ is just a submatrix of $(\I(\Gamma) \A)_{\Lambda^c}$.  Since $\|\I(\Gamma) \bPhi \x\|_2 = \| \bPhi^\Gamma \x \|_2$, this establishes the theorem.
\end{proof}

\section{Discussion}
\label{sec:disc}

\subsection{Robustness and stability}

Observe that we require roughly $O(D \log(N) )$ additional measurements to ensure that $\bPhi$ is $(\widetilde{M}, K, \delta)$-democratic compared to the number of measurements required to simply ensure that $\bPhi$ satisfies the RIP of order $K$.   This seems intuitive;  if we wish to be robust to the loss of any $D$ measurements while retaining the RIP of order $K$, then we should expect to take {\em at least} $D$ additional measurements. This is not unique to the CS framework.  For instance, by \emph{oversampling}, i.e., sampling faster than the minimum required Nyquist rate, uniform sampling systems can also improve robustness with respect to the loss of measurements.  However, a benefit of the democratic CS system is that the number of additional measurements needed grows more slowly than in the Nyquist case.   To see this, consider the case where we lose $D$ samples or measurements.  For a fixed time period, suppose that sampling the signal at the Nyquist rate yields $N$ samples.  To be robust to the loss of a contiguous block of $D$ samples, we must sample at $D+1$ times the Nyquist rate, yielding $DN$ additional samples.   In contrast, the number of additional measurements needed for a CS measurement system to be democratic is $O(D\log(N))$, given by (\ref{eq:Mdem}).  Thus, the number of additional samples required by a Nyqust sampler depends linearly on $D$ and $N$ while the number of additional measurements for democratic CS systems is still linear in $D$ but only {\em logarithmic} in $N$.  If $N$ is large, this can result in tremendous savings.  Note also that for a fixed $N$ and $K$, by driving $M$ higher a CS measurement system can be robust to the loss of a large fraction of the acquired measurements, whereas in Nyquist oversampling, the fraction of (consecutive) samples that can be dropped can never exceed $1/N$.

In some applications, this difference may have significant impact.  For example, in finite dynamic range quantizers, the measurements saturate when their magnitude exceeds some level.  Thus, when uniformly sampling with a low saturation level, if one sample saturates, then the likelihood that any of the neighboring samples will saturate is high, and significant oversampling may be required to ensure any benefit.  However, in CS, if many adjacent measurements were to saturate, then for only a slight increase in the number of measurements we can mitigate this kind of error by simply rejecting the saturated measurements; the fact that $\bPhi$ is democratic ensures that this strategy will be effective.

In addition to robustness, Theorem~\ref{thm:democracy} implies that reconstruction from a subset of CS measurements is \emph{stable} to the loss of a potentially larger number of measurements than anticipated.  To see this, suppose that and $M \times N$ matrix $\bPhi$ is $(M-D, K, \delta)$-democratic, but consider the situation where $D + \widetilde{D}$ measurements are dropped. It is clear from the proof of Theorem~\ref{thm:democracy} that if $\widetilde{D} < K$, then the resulting matrix $\bPhi^\Gamma$ will satisfy the RIP of order $K- \widetilde{D}$ with constant $\delta$.  Thus, from~\cite{CanRomTao::2006::Stable-signal}, if we define $\widetilde{K} = (K-\widetilde{D})/2$, then the reconstruction error is then bounded by
\begin{align}
\|\x - \widehat{\x}\|_{2} \leq C_{3}\frac{\|\x - \x_{\widetilde{K}}\|_{1}}{\sqrt{\widetilde{K}}},
\end{align}
where $\x_{\widetilde{K}}$ denotes the best $\widetilde{K}$-term approximation of $\x$ and $C_{3}$ is an absolute constant depending on $\bPhi$ that can be bounded using the constants derived in Theorem~\ref{thm:democracy}.  Thus, if $\widetilde{D}$ is small then the additional error caused by dropping too many measurements will also be relatively small.  In contrast, there is simply no analog to this kind of stability result for uniform sampling with linear reconstruction.  When the number of dropped samples exceeds $D$ (where $D$ represents the oversampling factor described above), there is are no guarantees as to the accuracy of the reconstruction.

\subsection{Numerical exploration}
As discussed previously, the democracy property is a stronger condition than the RIP.  To demonstrate this, we perform a numerical simulation which illustrates this point.  Specifically, we would like to compare the case where the measurements are dropped at random versus the case where the dropped measurements are selected by an adversary.  Ideally, we would like to know whether the resulting matrices satisfy the RIP.  Of course, this experiment is impossible to perform for two reasons: first, determining if a matrix satisfies the RIP is computationally intractable as it would require checking all possible $K$-dimensional sub-matrices of $\bPhi^\Gamma$.  Moreover, in the adversarial setting one would also have to search for the worst possible $\Gamma$ as well, which is impossible for the same reason.  Thus, we instead perform a far simpler experiment, which serves as a very rough proxy to the experiment we would like to perform.

The experiment proceeds over $100$ trials as follows.  We fix the parameters $N=2048$ and $K=13$ and vary $M$ in the range $(0,380)$.  In each trial we draw a new matrix $\Phi$ with $\phi_{ij}\sim\mathcal{N}(0, 1/M)$ and a new signal with $K$ nonzero coefficients, also drawn from a Gaussian distribution, and then the signal is normalized $\|\x\|_{2} = 1$. Over each set of trials we estimate two quantities:
\begin{enumerate}
\item the maximum $D$ such that we achieve exact reconstruction for a randomly selected  $(M-D)\times N$ submatrix of $\Phi$ on each of the $100$ trials;
\item the maximum $D$ such that we achieve exact reconstruction for $R=300$ randomly selected $(M-D)\times N$ submatrices of $\Phi$ on each of the $100$ trials..
\end{enumerate}
Ideally, the second case should consider \emph{all} $(M-D)\times N$ submatrices of $\Phi$ rather than just 300 submatrices, but as this is not possible (for reasons discussed above) we simply perform a random sampling of the space of possible submatrices.  Note also that exact recovery on one signal is also {\em not} proof that the matrix satisfies the RIP, although failure {\em is} proof that the matrix does not.

\begin{figure}[t] %  figure placement: here, top, bottom, or page
   \centering
   \includegraphics[width=.5\linewidth]{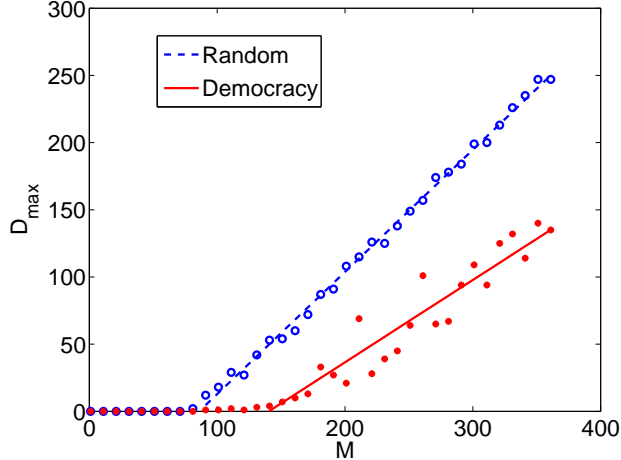}
   \caption{Maximum number of measurements that can be dropped $D_{\max}$ vs. number of measurements $M$ for (a) exact recovery of one $(M-D)\times N$ submatrix of $\bPhi$ and (b) exact recovery of $R=300$ $(M-D)\times N$ submatrices of $\bPhi$.}
   \label{fig:Ddemo}
\end{figure}

The results of this experiment are depicted in Figure~\ref{fig:Ddemo}.  The circles denote data points with the empty circles corresponding to the random selection experiment and the solid circles corresponding to the democracy experiment.  The lines denote the best linear fit for each data set where $D>0$, with the dashed line corresponding to the random selection experiment  and the solid line corresponding to democracy experiment.

The maximum $D$ corresponding to the random selection experiment grows linearly in $M$ (with coefficient 1) once the minimum number of measurements required for RIP, denoted by $M'$, is reached. This is because beyond this point at most $D = M-M'$ measurements can be discarded.  As demonstrated by the plot, $M' \approx 90$ for this experiment.  For the democracy experiment $M'\approx 150$, larger than for the RIP experiment.  Furthermore, the maximum $D$ for democracy grows more slowly than for the random selection case, which indicates that to be robust to the loss of {\em any} $D$ measurements, $C D$ additional measurements, with $C>1$, are actually {\em necessary}.

\section{Theorems}
\label{sec:thms}

In this section, we prove the two supporting Theorems used in the proof of Theorem~\ref{thm:democracy}.  We begin by demonstrating that the matrix $\A = [\I~ \bPhi]$ satisfies the RIP.  To do so, we first establish the following lemma, that closely parallels the result in equation (4.3) of~\cite{BarDavDeV::2008::A-Simple-Proof}.  The lemma demonstrates that for any $\u$, if we draw $\bPhi$ at random, then $\| \A \u \|_2$ is concentrated around $\|\u\|_2$.

\begin{lemma} \label{lem:conc}
Let $\bPhi$ by an $M\times N$ matrix with elements $\phi_{ij}$ drawn i.i.d. according to $\mathcal{N}(0, 1/M)$ and let $\A = [\I~ \bPhi]$. Furthermore, let $\u \in \mathbb{R}^{N+M}$ be an arbitrary vector with first $M$ entries denoted by $\w$ and last $N$ entries denoted by $\x$.  Let $\eta \in (0,1)$ be given.  Then
\begin{equation} \label{eq:exp}
\mathbb{E}\left( \| \A \u\|_2^2 \right) = \| \u \|_2^2
\end{equation}
and
\begin{equation} \label{eq:conc}
\mathbb{P}\left( \left| \| \A \u\|_2^2 - \|\u\|_2^2 \right| \ge  2 \eta \| \u\|_2^2 \right) \le 3 e^{-M \eta^2 / 8}.
\end{equation}
\end{lemma}
\begin{proof}
We first note that since $\A \u = \w + \bPhi \x$, we have that
\begin{align}
\| \A \u\|_2^2 & = \| \w + \bPhi \x \|_2^2 \notag \\
 & = (\w + \bPhi \x)^T (\w + \bPhi \x) \notag \\
 & = \w^T\w + 2 \w^T \bPhi \x + \x^T \bPhi^T \bPhi \x \notag \\
 & = \|\w\|_2^2 + 2 \w^T \bPhi \x + \| \bPhi \x \|_2^2. \label{eq:split}
\end{align}
Since the entries $\phi_{ij}$ are i.i.d.\ according to $\mathcal{N}(0,1/M)$, it is straightforward to show that $\mathbb{E} \left( \| \bPhi \x \|_2^2 \right) = \|\x\|_2^2$ (see, for example,~\cite{DasGup::1999::An-Elementary-Proof}).  Similarly, one can also show that $2 \w^T \bPhi \x \sim \mathcal{N} \left( 0,4 \| \w \|_2^2 \|\x\|_2^2 / M\right)$, since the elements of $\bPhi\x$ are distributed as zero mean Gaussian variables with variance $\|\x\|_{2}^{2}/M$.   Thus, from (\ref{eq:split}) we have that
$$
\mathbb{E} \left( \| \A \u\|_2^2 \right) = \| \w \|_2^2 + \| \x \|_2^2 ,
$$
and since $\| \u \|_2^2 = \| \w \|_2^2 + \| \x \|_2^2 + $, this establishes (\ref{eq:exp}).

We now turn to (\ref{eq:conc}).  Using the arguments in~\cite{DasGup::1999::An-Elementary-Proof}, one can show that
\begin{equation} \label{eq:conc1}
\mathbb{P}\left( \left| \| \bPhi \x\|_2^2 - \|\x\|_2^2 \right| \ge \eta \| \x \|_2^2 \right) \le 2 e^{-M \eta^2 / 8}.
\end{equation}
As noted above, $2 \w^T \bPhi \x \sim \mathcal{N} \left( 0,4 \| \w \|_2^2 \|\x\|_2^2 / M\right)$.  Hence, we have that
\begin{align}
\mathbb{P} \left( \left| 2 \w^T \bPhi \x \right| \ge \eta  \| \w \|_2 \| \x\|_2 \right) =~ & 2 Q \left( \frac{ \eta \| \w \|_2 \| \x\|_2 }{2 \| \w \|_2 \|\x\|_2  / \sqrt{M}} \right) \notag\\
=~ & 2 Q (\sqrt{M} \eta / 2) \notag,
\end{align}
where $Q(\cdot)$ denotes the tail integral of the standard Gaussian distribution.  From (13.48) of \cite{JKB-CUDV1} we have that
$$
Q(z) \le \frac{1}{2} e^{-z^2/2}
$$
and thus we obtain
\begin{equation} \label{eq:conc2}
\mathbb{P} \left( \left| 2 \w^T \bPhi \x \right| \ge \eta  \| \w \|_2 \| \x\|_2 \right) \le  e^{- M \eta^2 / 8}.
\end{equation}
Thus, combining (\ref{eq:conc1}) and (\ref{eq:conc2}) we obtain that with probability at least $1 - 3 e^{-M \eta^2 / 8}$ we have that both
\begin{equation} \label{eq:conc3}
(1- \eta)\| \x \|_2^2 \le \| \bPhi \x\|_2^2 \le (1+\eta) \| \x \|_2^2
\end{equation}
and
\begin{equation} \label{eq:conc4}
- \eta \| \w \|_2 \| \x \|_2 \le 2 \w^T \bPhi \x \le \eta \| \w \|_2 \| \x\|_2.
\end{equation}
Using (\ref{eq:split}), we can combine (\ref{eq:conc3}) and (\ref{eq:conc4}) to obtain
\begin{align*}
\| \A \u\|_2^2 & \le \| \w \|_2^2 + \eta \| \w \|_2 \| \x\|_2 + (1+\eta) \| \x \|_2^2 \\
 & \le (1+\eta) \left( \| \w \|_2^2 + \| \x \|_2^2 \right) + \eta \| \w \|_2 \| \x\|_2 \\
 & \le (1+\eta) \|\u\|_2^2 + \eta \|\u\|_2^2 \\
 & = (1+2\eta) \| \u \|_2^2,
\end{align*}
where the last inequality follows from the fact that $\| \w \|_2 \| \x \|_2  \le \| \u \|_2 \| \u \|_2$.  Similarly, we also have that
$$
\| \A \u\|_2^2 \ge (1-2\eta) \| \u \|_2^2,
$$
which establishes (\ref{eq:conc}).
\end{proof}

We note that while the above proof assumes that the entries of $\bPhi$ are Gaussian, this proof holds with essentially no modifications for a wide class of {\em sub-Gaussian}
distributions.  A random variable $X$ is sub-Gaussian if there exists a constant $C > 0$ such that
\begin{equation} \label{eq:sgdef}
\mathbb{E} \left( e^{X t} \right) \le e^{ C^2t^2/2}
\end{equation}
for all $t \in \mathbb{R}$.  This says that the moment-generating
function of our distribution is dominated by that of a Gaussian
distribution, which is also equivalent to requiring that the tails
of our distribution decay {\em at least as fast} as the tails of a
Gaussian distribution. Examples of sub-Gaussian distributions
include the Gaussian distribution, the Rademacher distribution, and
the uniform distribution.  In general, any distribution with bounded support is sub-Gaussian.  See~\cite{SG} for more discussion on sub-Gaussian random variables.  It can be shown (see Lemma 6.1 of~\cite{DeVoreL1IO} or~\cite{SubGaussCNX}) that if the entries of $\bPhi$ are drawn according to a sub-Gaussian distribution, then (\ref{eq:conc1}) holds where 8 is replaced with a constant that depends on the constant $C$ in (\ref{eq:sgdef}).  Similarly, the cross-term $\left| 2 \w^T \bPhi \x \right| $ is also a sub-Gaussian random variable, and so using elementary results in~\cite{SG}, a bound analogous to (\ref{eq:conc2}) can be obtained.

Using Lemma~\ref{lem:conc}, we now demonstrate that the matrix $\A$ satisfies the RIP provided that $M$ is sufficiently large.  %This theorem follows immediately from Lemma~\ref{lem:conc} by using a proof identical to that given in~\cite{BarDavDeV::2008::A-Simple-Proof}, so we omit the proof for the sake of brevity.

\begin{theorem}
\label{thm:rip}
Let $\bPhi$ be an $M\times N$ matrix with elements $\phi_{ij}$ drawn according to $\mathcal{N}(0, 1/M)$ and let let $\A = [\I~ \bPhi]$.  If
\begin{align}
\label{eq:justice}
M = C_1 (K + D)\log\left(\frac{N+M}{K + D}\right)
\end{align}
then with probability exceeding $1-3 e^{-C_2 M}$ we have that $\A$ satisfies the RIP of order $(K+D)$ with constant $\delta$, where $C_1$ is arbitrary and $C_2  = (\delta/8)^2 - \log(42 e/ \delta)/C_1.$
\end{theorem}
\begin{proof}
First note that it is enough to prove (\ref{eq:justice}) in the case
$\|\vx\|_2=1$, since $\A$ is linear.  Next,
fix an index set $J \subset \{1, 2, \ldots, N+M\}$ with $|J| = K+D$, and
let $X_J$ denote the $(K+D)$-dimensional subspace spanned by the columns
of $\A$ indexed by $J$. We choose a finite set of points $S_J$
such that $S_J \subseteq X_J$, $\|\vs\|_2\le 1$ for all $\vs \in
S_J$, and for all $\vx \in X_J$ with $\|\vx\|_2\le 1$ we have
\begin{equation} \label{have1}
\min_{\vs \in S_J}\|\vx- \vs\|_2 \le \epsilon.
\end{equation}
One can show (see Chapter 15 of \cite{Lorentz}) that such a set $S_J$
exists with $|S_J| \le (3/\epsilon)^{K+D}$. We then repeat this process
for each possible index set $J$, and collect all the sets $S_J$
together
\begin{equation}
S = \bigcup_{J : |J|=K+D} S_J.
\end{equation}
There are $\binom{N+M}{K+D} \le \left(e\frac{N+M}{K+D}\right)^{K+D} $ possible index sets $J$, and
hence $|S| \le \left(\frac{3 e}{\epsilon} \frac{N+M}{K+D}\right)^{K+D}$.  We now use the union bound to apply Lemma~\ref{lem:conc} to this set of points such that, with probability exceeding
\begin{equation}
\label{eq:unionprob}
1 - 3 \left(\frac{3 e}{\epsilon} \frac{N+M}{K+D}\right)^{K+D} e^{-\frac{M \eta^2}{8}} = 1 - 3 e^{- \frac{M \eta^2}{8} + (K+D) \log \left(\frac{3 e}{\epsilon} \frac{N+M}{K+D}\right)},
\end{equation}
we have
\begin{equation*}
(1 - 2\eta) \|\vs\|_2^2 \le \|\A \vs\|_2^2 \le (1+ 2\eta)
\|\vs\|_2^2,\quad {\rm for \ all}\ \vs\in S.
\end{equation*}

We now define $\Sigma_{K+D} = \{ \x : \|\x\|_0 \le K+D \}$.  We define $B$ as the smallest number such that
\begin{equation} \label{eq:BBound}
\|\A \vx \|^2_2 \le B \|\vx\|^2_2,\quad {\rm for \
all}\ \vx \in \Sigma_{K+D}, \ \|\vx\|_2 \le 1.
\end{equation}
Our goal is to show that $B \le \sqrt{1+\delta}$.  For this, we recall that for
any $\vx \in \Sigma_{K+D}$ with $\|\vx\|_2 \le 1$, we can pick a $\vs
\in S$ such that $\|\vx - \vs\|_2 \le \epsilon$ and such that $\vx -
\vs \in \Sigma_{K+D}$ (since if $\vx \in X_J$, we can pick $\vs \in S_J
\subset X_J$ satisfying $\|\vx - \vs\|_2 \le \epsilon$). In this
case we have
\begin{equation*}
\|\A \vx \|_2 \le \|\A \vs \|_2+
\|\A(\vx-\vs)\|_2 \le \sqrt{1 + 2\eta} + \sqrt{B}
\epsilon.
\end{equation*}
Since by definition $B$ is the smallest number for which (\ref{eq:BBound}) holds, we obtain $\sqrt{B} \le
\sqrt{1 + 2\eta}+\sqrt{B}  \epsilon$, which upon rearranging yields $\sqrt{B} \le \sqrt{1 + 2\eta}/(1-\epsilon)$.  One can show that by setting $\epsilon = \delta/14$ and $\eta = \delta/2 \sqrt{2}$, we have that $\sqrt{1 + 2\eta}/(1-\epsilon) \le \sqrt{1+\delta}$, which establishes the upper inequality in (\ref{eq:RIP}).  The lower
inequality follows from this since
\begin{equation*}
\|\A \vx \|_2 \ge \|\A \vs\|_2- \|\A(\vx-\vs)\|_2 \ge \sqrt{1 - 2\eta} - \sqrt{1 + \delta} \frac{\epsilon}{1-\epsilon}  \geq \sqrt{1-\delta},
\end{equation*}
where the last inequality again holds with $\epsilon = \delta/14$ and $\eta = \delta/2 \sqrt{2}$.  This establishes the theorem.  To arrive at the formula for $C_2$ we first bound the result in (\ref{eq:unionprob}) using
$$
\log \left( \frac{3 e}{\epsilon} \frac{N+M}{K+D} \right) \le \log \left( \frac{3 e}{\epsilon} \right) \log \left( \frac{N+M}{K+D} \right)
$$
and then we replace $(K+D) \log( (N+M)/(K+D) )$ with $M/C_1$.  After simplification, this yields $C_2 = \eta^2/8 - \log(3 e / \epsilon) / C_1$.  By substituting the values for $\epsilon$ and $\eta$, we obtain the desired result.
\end{proof}

In Theorem~\ref{lem:PRIP} below, we show that the matrix $\Proj_{\Lambda}^{\perp}\A$ satisfies a modified version of the RIP.   We begin with an elementary lemma that is a straightforward generalization of Lemma 2.1 of~\cite{CandesRIP}, and states that RIP operators approximately preserve inner products between sparse vectors.

\begin{lemma} \label{lem:ip}
Let $\u,\v \in \reals^N$ be given, and suppose that a matrix $\A$
satisfies the RIP of order $\max(\|\u+\v\|_0,\|\u-\v\|_0)$ with isometry
constant $\delta$. Then
\begin{equation} \label{eq:ip-bound}
\left| \langle \A \u, \A \v \rangle  - \langle \u , \v \rangle
\right| \le \delta \|\u\|_2 \|\v\|_2.
\end{equation}
\end{lemma}
\begin{proof}
We first assume that $\|\u\|_2 = \|\v\|_2 = 1$.  From the fact that
$$
\|\u \pm \v\|_2^2 = \|\u\|_2^2 + \|\v\|_2^2 \pm 2 \langle \u,\v\rangle = 2
\pm 2\langle \u, \v \rangle$$ and since $\A$ satisfies the RIP, we
have that
$$
(1-\delta)(2 \pm 2\langle \u,\v \rangle) \le \|\A \u \pm \A \v \|_2^2 \le (1+\delta)(2 \pm 2\langle \u,\v \rangle).
$$
From the parallelogram identity we obtain
\begin{eqnarray}
\langle \A \u, \A \v \rangle
&=& \frac{1}{4} \left( \|\A \u +
\A \v \|_2^2 - \|\A \u - \A \v \|_2^2 \right) \\
&\le& \frac{(1 +
\langle \u,\v \rangle)(1+\delta) - (1 -  \langle \u,\v
\rangle)(1-\delta)}{2}
 = \langle \u,\v \rangle + \delta.
\end{eqnarray}
Similarly, one can show that $\langle \A \u, \A \v \rangle \ge
\langle \u,\v \rangle - \delta$, and thus
$|\langle \A \u, \A \v \rangle - \langle \u,\v \rangle| \le \delta$.  The result follows for
$\u$, $\v$ with arbitrary norm from the bilinearity of the inner product.
\end{proof}

One consequence of this result is that sparse vectors that are orthogonal in $\reals^N$ remain nearly orthogonal after the
application of $\A$. From this observation, it was demonstrated
independently in~\cite{SPARS} and~\cite{sp} that if $\A$ has the
RIP, then $\Proj^\perp_{\Lambda} \A$ satisfies a modified version of the RIP.

\begin{theorem} \label{lem:PRIP}
Suppose that $\A$ satisfies the RIP of order $K$ with isometry
constant $\delta$, and let $\Lambda \subset \{1, 2, \ldots, N\}$. Define $\Proj^\perp_\Lambda$ as in(\ref{eq:Pperpdef}). If
$|\Lambda| < K$ then
\begin{equation} \label{eq:PRIP}
\left(1 - \frac{\delta}{1-\delta} \right) \|\u\|_2^2 \le \| \Proj^\perp_{\Lambda} \A \u\|_2^2 \le (1+\delta) \|\u\|_2^2
\end{equation}
for all $\u \in \reals^N$ such that $\|\u\|_0 \le K - |\Lambda|$ and
$\supp(\u) \cap \Lambda =  \emptyset$.
\end{theorem}
\begin{proof}
From the definition of $\Proj^\perp_{\Lambda} \A$ in (\ref{eq:Pdef}), we may decompose $\Proj^\perp_{\Lambda} \A \u$ as
$\Proj^\perp_{\Lambda} \A \u = \A \u - \Proj_\Lambda \A \u$. Since $\Proj_\Lambda$ is an
orthogonal projection,  we can write
\begin{equation} \label{eq:triangle}
\|\A \u\|_2^2 = \|\Proj_\Lambda \A \u\|_2^2 + \|\Proj^\perp_{\Lambda} \A \u\|_2^2.
\end{equation}
Our goal is to show that $\| \A \u \|_2 \approx \|\Proj^\perp_{\Lambda} \A \u\|_2$,
or equivalently, that $\|\Proj_\Lambda \A \u\|_2$ is small. Towards
this end, we note that since $\Proj_\Lambda \A \u$ is orthogonal to
$\Proj^\perp_{\Lambda} \A \u$, we have
\begin{align}
\langle \Proj_\Lambda \A \u, \A \u \rangle & = \langle \Proj_\Lambda \A \u, \Proj_\Lambda \A \u + \Proj^\perp_{\Lambda} \A \u \rangle \notag \\
 & = \langle \Proj_\Lambda \A \u, \Proj_\Lambda \A \u \rangle + \langle \Proj_\Lambda \A \u, \Proj^\perp_{\Lambda} \A \u \rangle \notag \\
 & = \|\Proj_\Lambda \A \u\|_2^2. \label{eq:cosine}
\end{align}
Since $\Proj_\Lambda$ is a projection onto $\Ran(\A_\Lambda)$ there
exists a $\z \in \real^N$ with $\supp(\z) \subseteq \Lambda$ such that $\Proj_\Lambda \A \u = \A
\z$. Furthermore, by assumption, $\supp(\u) \cap \Lambda = \emptyset$.
Hence $\langle \u, \z \rangle = 0$ and from the RIP and Lemma
\ref{lem:ip},
$$
\frac{|\langle \Proj_\Lambda \A \u, \A \u \rangle|}{\|\Proj_\Lambda \A \u\|_2\|\A \u\|_2} = \frac{|\langle \A \z, \A \u \rangle|}{\| \A \z\|_2\|\A \u\|_2}  \le  \frac{|\langle \A \z, \A \u \rangle|}{(1-\delta) \| \z\|_2\|\u\|_2} \le \frac{\delta}{1-\delta}.
$$
Combining this with (\ref{eq:cosine}), we obtain
$$
\| \Proj_\Lambda \A \u\|_2 \le \frac{\delta}{1-\delta} \|\A \u\|_2.
$$
Since we trivially have that $\|\Proj_\Lambda \A \u\|_2 \ge 0$, we can
combine this with (\ref{eq:triangle}) to obtain
$$
\left(1- \left(\frac{\delta}{1-\delta}\right)^2\right)\|\A \u\|_2^2
\le \| \Proj^\perp_{\Lambda} \A \u\|_2^2 \le \|\A \u\|_2^2.
$$
Since $\|\u\|_0 \le K$, we can use the RIP to obtain
$$
\left(1- \left(\frac{\delta}{1-\delta}\right)^2\right) (1-\delta)\|\u\|_2^2 \le \| \Proj^\perp_{\Lambda} \A \u\|_2^2 \le (1+\delta)\|\u\|_2^2,
$$
which simplifies to (\ref{eq:PRIP}).
\end{proof}

\bibliographystyle{IEEEbib}
\footnotesize
\bibliography{samptaRefs}

\end{document}